\documentclass[12pt,reqno]{amsart}
\usepackage{amsfonts}
\usepackage{amssymb,color}
\usepackage{amssymb}

\usepackage[colorlinks=true]{hyperref}
\hypersetup{urlcolor=blue, citecolor=blue}

\usepackage[margin=2.5cm, a4paper]{geometry}

\newtheorem{theo}{Theorem}[section]
\newtheorem{lemm}[theo]{Lemma}
\newtheorem{defi}[theo]{Definition}

\numberwithin{equation}{section}

\newcommand{\bal}{\begin{align}}
\newcommand{\bbal}{\begin{align*}}
\newcommand{\beq}{\begin{equation}}
\newcommand{\eeq}{\end{equation}}
\newcommand{\bca}{\begin{cases}}
\newcommand{\eca}{\end{cases}}

\newcommand{\pa}{\partial}
\newcommand{\fr}{\frac}
\newcommand{\na}{\nabla}
\newcommand{\De}{\Delta}
\newcommand{\de}{\delta}

\newcommand{\cd}{\cdot}

\newcommand{\dd}{\mathrm{d}}
\newcommand{\B}{\dot{B}}
\newcommand{\FB}{\dot{FB}}
\newcommand{\LL}{\tilde{L}}
\newcommand{\R}{\mathbb{R}}

\newcommand{\J}{\mathcal{J}}
\newcommand{\Z}{\mathbb{Z}}
\newcommand{\e}{\vec{e}}
\newcommand{\ee}{\vec{e_1}}

\newcommand{\F}{\dot{FB}}

\begin{document}


\title[Ill-posedness for the Hamilton-Jacobi equation in Besov spaces $B^0_{\infty,q}$]{Ill-posedness for the Hamilton-Jacobi equation in Besov spaces $B^0_{\infty,q}$}

\author[J. Li]{Jinlu Li}
\address{School of Mathematics and Computer Sciences, Gannan Normal University, Ganzhou 341000, China \& Department of Mathematics, Sun Yat-sen University, Guangzhou 510275, China}
\email{lijl29@mail2.sysu.edu.cn}

\author[W. Zhu]{Weipeng Zhu}
\address{Department of Mathematics, Sun Yat-sen University, Guangzhou 510275, China}
\email{mathzwp2010@163.com}

\author[Z. Yin]{Zhaoyang Yin}
\address{Department of Mathematics, Sun Yat-sen University, Guangzhou 510275, China \& Faculty of Information Technology,  Macau University of Science and Technology, Macau, China}
\email{mcsyzy@mail.sysu.edu.cn}

\begin{abstract}
In this paper, we study the Cauchy problem for the following Hamilton-Jacobi equation
\bbal\bca
\pa_tu-\De u=|\na u|^2,\quad t>0, \ x\in \R^d,\\
u(0,x)=u_0, \quad \quad x\in \R^d.
\eca\end{align*}
We show that the solution map in Besov spaces $B^0_{\infty,q}(\R^d),1\leq q\leq \infty$ is discontinuous at origin. That is, we can construct a sequence initial data $\{u^N_0\}$ satisfying $||u^N_0||_{B^0_{\infty,q}(\R^d)}\rightarrow 0, \ N\rightarrow \infty$ such that the corresponding solution $\{u^N\}$ with $u^N(0)=u^N_0$ satisfies
\bbal
||u^N||_{L^\infty_T(B^0_{\infty,q}(\R^d))}\geq c_0, \qquad \forall \ T>0, \quad N\gg 1,
\end{align*}
with a constant $c_0>0$ independent of $N$.
\end{abstract}

\maketitle

\section{Introduction and main result}

In our paper, we study the viscous Hamilton-Jacobi equation,
\beq\bca\label{H-J}
\pa_tu-\De u=|\na u|^2,\quad t>0, \ x\in \R^d,\\
u(0,x)=u_0, \quad \quad x\in \R^d.
\eca\eeq
where $d\geq 1$, $\na=(\pa_{x_1},\cdots,\pa_{x_d})$.

Note that the problem \eqref{H-J} is the special form of the following viscous Hamilton-Jacobi equation,
\beq\bca\label{H-J-p}
\pa_tu-\De u=|\na u|^p,\quad t>0, \ x\in \R^d,\\
u(0,x)=u_0, \quad \quad x\in \R^d.
\eca\eeq
where $d\geq 1$, $p \geq 1$.

The viscous Hamilton-Jacobi equation owns both mathematical and physical interest. Indeed, from mathematical point of view, it is the simplest example of a parabolic PDE with a nonlinearity depending only on the first order spatial derivatives of $u$, and it describes a model for growing random interfaces, which is known as the Kardar-Parisi-Zhang equation (see \cite{KPZ86,KS88}).

An equation of the type \eqref{H-J-p} has attracted much attention \cite{A97,AB98,B92,BL99,BSW02,GGK03} in recent years. In the earliest of these papers \cite{B92}, under the assumptions that $p=1$ and $u_0\in C^3_0(\R^d)$, results on the existence and uniqueness of a classical solution was obtained. Subsequently, under the hypothesis that $u_0\in C^2_0(\R^d)$ and $p \geq 1$, the existence of a unique local (and actually global) classical solution of problem \eqref{H-J-p} was established in \cite{AB98}. This result was recently extended to $u_0\in C_0(\R^d)$ and $p>0$ in \cite{GGK03}. Under the much weaker assumptions on $u_0$, the existence of a suitably-defined weak solution when $1\leq p<2$
and $u_0$ is a Radon measure was investigated in \cite{A97}. In \cite{BL99}, Benachour and  Lauren\c{c}ot investigated the existence and uniqueness of nonnegative weak solutions to problem \eqref{H-J-p} with initial data $u_0$ in the space of bounded and non-negative measures when $1<p<(d+2)/(d+1)$ and in $L^1(\R^d) \cap L^q(\R^d)$ when $(d+2)/(d+1) \leq p<2$ provided that $q$ is large enough. The result was generalized by Ben-Artzi, Souplet and Weissler in \cite{BSW02}.
For more results of Hamilton-Jacobi equation, we refer the readers to see \cite{BKL04,I05,IK17,KW08,LS03,S11}.

To solve the original equations, we may consider the following integral equations:
$$u(t)=e^{t\De}u_0+\int^t_0e^{(t-\tau)\De}(|\na u|^2)\dd \tau,$$
where
$$e^{t\De}f=\mathcal{F}^{-1}[e^{-|\xi|^2t}\hat{f}(\xi)].$$
Then we can define the maps $A_n$ for $n=1,2,\cdots$ by the recursive formulae
\bbal
&A_1(f):=e^{t\De}f,
\\&A_n(f):=\sum\limits_{n_1, n_2\geq 1,n_1+n_2=n}\sum^d_{i=1}\int^t_0e^{(t-\tau)\De}[\pa_iA_{n_1}(f)][\pa_iA_{n_2}(f)]\dd \tau \quad \mathrm{for} \quad n\geq 2.
\end{align*}

In this paper, motivated by \cite{BP,IK,Wbao}, we can construct a sequence $f^N$ such that $||f^N||_{B^0_{\infty,q}}\rightarrow 0$ when $N$ tends to infinity. Here, it is easy to check that
\bbal
||A_1(f^N)||_{B^0_{\infty,q}(\R^d)}\rightarrow 0, \qquad N\rightarrow \infty.
\end{align*}
The main key point, we can show that
\bbal
||A_2(f^N)||_{B^0_{\infty,q}(\R^d)}\geq c\de^2 \ (\mathrm{or} \ c\de^3), \qquad N\gg1,\ \de\ll 1,
\end{align*}
for a small constant $c$ independent of $N$ and $q$. Furthermore, we also deduce that
\bbal
||u[f^N]-A_1(f^N)-A_2(f^N)||_{B^0_{\infty,q}(\R^d)}\leq C\de^3 \ (\mathrm{or} \ C\de^4), \qquad N\gg1,\ \de\ll 1,
\end{align*}
for a big constant $C$ independent of $N$ and $q$. This show that \eqref{H-J} is ill-posedness in Besov spaces $B^0_{\infty,q}(\R^d)$.

Our main ill-posedness result for Hamilton-Jacobi equation in Besov spaces reads as follows:
\begin{theo}\label{th1.3}
Let $d\geq 1$ and $s>1+\frac d2$. For $1\leq q\leq \infty$, Eq. \eqref{H-J} is ill-posedness in $B^0_{\infty,q}(\R^d)$ in the sense that the solution map is discontinuous at origin. More precisely, there exist a sequence of initial data $\{f^N\}_{N\geq 1}\subset H^s(\R^d)$ and a sequence of time $\{t_N\}_{N\geq 1}$ with
\bbal
||f^N||_{B^0_{\infty,q}(\R^d)}\rightarrow 0,  \quad  t_N\rightarrow 0, \quad N\rightarrow \infty,
\end{align*}
such that the corresponding sequence solutions $u[f^N]\in \mathcal{C}([0,\infty);H^s(\R^d))$ to initial value problem \eqref{H-J} with $u[f^N](0)=f^N$ satisfies
\bbal
||u[f^N](t_N)||_{B^0_{\infty,q}(\R^d)}\geq c_0,
\end{align*}
with a positive constant $c_0$ independent of $N$.
\end{theo}

Our paper is organized as follows. In Section 2, we give some preliminaries which will be used in the sequel. In Section 3, we will give the proof of Thorem \ref{th1.3}.\\

\noindent\textbf{Notation.} In the following, since all spaces of functions are over $\mathbb{R}^d$, for simplicity, we drop $\mathbb{R}^d$ in our notations of function spaces if there is no ambiguity. Let $C\geq 1$ and $c\leq 1$ denote constants which can be different at different place. We use $A\lesssim B$ to denote $A\leq CB$.

\section{Preliminaries}

In this section we collect some preliminary definitions and lemmas. For more details we refer the readers to \cite{B.C.D}.

Let $\varphi: {\mathbb R}^d\to [0, 1]$ be a radial and smooth function satisfying $\mathrm{Supp }\ \varphi\subset\mathcal{C}\triangleq \{\xi\in\mathbb{R}^d:\frac 3 4\leq|\xi|\leq \frac 8 3\}$, $\varphi\equiv 1$ for $\frac43\leq |\xi|\leq \frac32$ and
\bbal
\sum_{j\in \Z}\varphi(2^{-j}\xi)=1,
\end{align*}
for all $\xi\neq 0$. We set $\Psi(\xi)=1-\sum_{j\geq 0}\varphi(2^{-j}\xi)$. For $u \in \mathcal{S}'$, $q\in {\mathbb Z}$, we define the Littlewood-Paley operators: $\dot{\Delta}_q{u}=\mathcal{F}^{-1}(\varphi(2^{-q}\cdot)\mathcal{F}u)$, ${\Delta}_q{u}=\dot{\Delta}_q{u}$ for $q\geq 0$, ${\Delta}_q{u}=0$ for $q\leq -2$ and $\Delta_{-1}u=\mathcal{F}^{-1}(\Psi(\xi) \mathcal{F}u)$, and $\dot{S}_q{u}=\mathcal{F}^{-1}\big(\Psi(2^{-q}\xi)\mathcal{F}u\big)$. Here we use ${\mathcal{F}}(f)$ or $\widehat{f}$ to denote
the Fourier transform of $f$.

Applying the above decomposition, the Besov spaces $B^s_{p,r}(\R^d)$, the homogeneous Fourier--Besov spaces $\F_{p,r}^s(\mathbb{R}^{3})$ and the Chemin--Lerner type spaces $ \tilde{L}^{\lambda}(0,T; \F^{s}_{p,r}(\mathbb{R}^{3}))$ can be defined as follows:

\begin{defi}(\cite{B.C.D})
Let $s\in \mathbb{R}$ and $1\leq p,r\leq\infty$, the Besov space $B^s_{p,r}$ or $\dot{B}^s_{p,r}$ is defined by
\bbal
B^s_{p,r}(\R^d)=\Big\{f\in \mathcal{S}'(\R^d):||f||_{B^s_{p,r}}:=\big|\big|(2^{js}\|\Delta_j{u}\|_{L^p})_{j\in {\mathbb Z}}\big|\big|_{\ell^r}<\infty\Big\},
\end{align*}
or
\bbal
\dot{B}^s_{p,r}(\R^d)=\Big\{f\in \mathcal{S}'_h(\R^d):||f||_{\dot{B}^s_{p,r}}:=\big|\big|(2^{js}\|\dot{\Delta}_j{u}\|_{L^p})_{j\in {\mathbb Z}}\big|\big|_{\ell^r}<\infty\Big\}.
\end{align*}
Here, $u\in\mathcal{S}'_h$ denotes $u\in S'$ and $\lim\limits_{j\rightarrow -\infty}||\dot{S}_ju||_{L^\infty}=0$.
\end{defi}

\begin{defi}\label{def2.1}(\cite{IK,KY11})
Let $s\in \mathbb{R}$ and $1\leq p,r\leq\infty$, the space $\F^{s}_{p,r}(\mathbb{R}^{d})$ is defined to be the set of all tempered distributions $f\in \mathcal{S}'(\mathbb{R}^{d})$ such that $\hat{f} \in L_{loc}^1(\R^d)$ and the following norm is finite:
\begin{equation*}
\|f\|_{\F^{s}_{p,r}}:= \begin{cases} \left(\sum_{j\in\mathbb{Z}}2^{jsr}\|\widehat{\dot{\Delta}_{j}f}\|_{L^{p}}^{r}\right)^{\frac{1}{r}}
\ \ &\text{if}\ \ 1\leq r<\infty,\\
\sup_{j\in\mathbb{Z}}2^{js}\|\widehat{\dot{\Delta}_{j}f}\|_{L^{p}}\ \
&\text{if}\ \
r=\infty.
\end{cases}
\end{equation*}
\end{defi}

\begin{defi}\label{def2.2} For $0<T\leq\infty$, $s\in \mathbb{R}$ and $1\leq p, r, \lambda\leq\infty$, we set  {\rm{(}}with the usual convention if
$r=\infty${\rm{)}}:
$$
\|f\|_{\tilde{L}^{\lambda}_{T}(\F^{s}_{p,r})}:=\big(\sum_{j\in\mathbb{Z}}2^{jsr}\|\widehat{\dot{\Delta}_{j}f}\|_{L^{\lambda}(0,T;
	L^{p})}^{r}\big)^{\frac{1}{r}}.
$$
We then define the space $\tilde{L}^{\lambda}(0,T; \F^{s}_{p,r}(\mathbb{R}^{d}))$  as the set of temperate distributions $f$ over $(0,T)\times \mathbb{R}^{3}$ such that $\lim\limits_{j\rightarrow -\infty}S_{j}f=0$ in $\mathcal{S}'((0,T)\times\mathbb{R}^{d})$ and $\|f\|_{\tilde{L}^{\lambda}_{T}(\F^{s}_{p,r})}<\infty$.
\end{defi}

Next we recall Bony's decomposition from \cite{B.C.D}.
$$uv=\dot{T}_uv+\dot{T}_vu+\dot{R}(u,v),$$
with
\[\dot{T}_uv\triangleq \sum\limits_j \dot{S}_{j-1}u\dot{\Delta}_j v, \quad \dot{R}(u,v)\triangleq \sum_{j}\sum\limits_{|k-j|\leq1}\dot{\Delta}_j u \dot{\Delta}_k v.\]
This is now a standard tool for nonlinear estimates. Now we use Bony's decomposition to prove some nonlinear estimates which will be used in the proof of our theorem.

\begin{lemm}\label{le2.1}
Let $1\leq r\leq 2$. Then, there hold
\bbal
||\dot{T}_uv||_{\FB^{0}_{1,r}}\leq C||u||_{\FB^{-1}_{1,r}}||v||_{\FB^{1}_{1,r}}, \qquad
||\dot{R}(u,v)||_{\FB^{0}_{1,r}}\leq C ||u||_{\FB^{-1}_{1,r}}||v||_{\FB^{1}_{1,r}},
\end{align*}
and
\bbal
&||\dot{T}_uv||_{\LL^1_T(\FB^{0}_{1,r})}\leq C||u||_{\LL^\infty_T(\FB^{-1}_{1,r})}||v||_{\LL^1_T(\FB^{1}_{1,r})},
\\& ||\dot{R}(u,v)||_{\LL^1_T(\FB^{0}_{1,r})}\leq C ||u||_{\LL^\infty_T(\FB^{-1}_{1,r})}||v||_{\LL^1_T(\FB^{1}_{1,r})}.
\end{align*}
\end{lemm}
\begin{proof}
Note that
\bbal
||\sum_{j\in \Z}\dot{S}_{j-1}u\dot{\De}_jv||_{\FB^0_{1,r}}&\leq ||\sum_{j\in \Z}\dot{S}_{j-1}u\dot{\De}_jv||_{\FB^0_{1,1}} \\&\leq C\sum_{j\in\Z}||\widehat{\dot{S}_{j-1}u\dot{\De}_jv}||_{L^1}
\\&\leq \sum_{j\in\Z}(2^{-j}||\widehat{\dot{S}_{j-1}u}||_{L^1})\cd (2^{j}||\widehat{\dot{\De}_jv}||_{L^1})
\\&\leq C||u||_{\FB^{-1}_{1,r}}||v||_{\FB^{1}_{1,r}},
\end{align*}
and
\bbal
||\sum_{j\in \Z}\dot{\De}_ju\widetilde{\dot{\De}_j}v||_{\FB^0_{1,r}}&\leq ||\sum_{j\in \Z}\dot{\De}_ju\widetilde{\dot{\De}_j}v||_{\FB^0_{1,1}} \\&\leq
||\sum_{j\in \Z}\widehat{\dot{\De}_ju\widetilde{\dot{\De}_j}v}||_{L^1}
\\&\leq \sum_{j\in \Z}\sum_{|i-j|\leq 1}||\widehat{\dot{\De}_ju}||_{L^1}||\widehat{\dot{\De}_iv}||_{L^1}
\\&\leq C(\sum_{j\in \Z}2^{-2j}||\widehat{\dot{\De}_ju}||^2_{L^1})^{\frac12}||v||_{\FB^{1}_{1,2}}
\\&\leq C||u||_{\FB^{-1}_{1,r}}||v||_{\FB^{1}_{1,r}}.
\end{align*}
This completes the proof.
\end{proof}

\begin{lemm}\label{le2.2}
Let $d\geq 1$ and $s>1+\frac d2$. Assume that $u\in \mathcal{C}([0,T];H^s)$ be the solution of Eq. \eqref{H-J}. Then, we have for $i=1,2\cdots d$,
\bbal
& \quad ||\int^t_0e^{(t-\tau)\De}(\pa_iu\pa_i v)\dd \tau||_{\LL^\infty_T(\FB^0_{1,2})}+||\int^t_0e^{(t-\tau)\De}(\pa_iu\pa_i v)\dd \tau||_{\LL^1_T(\FB^2_{1,2})}
\\&\leq C\big(||u||_{\LL^\infty_T(\FB^0_{1,2})}||v||_{\LL^1_T(\FB^2_{1,2})}+||v||_{\LL^\infty_T(\FB^0_{1,2})}||u||_{\LL^1_T(\FB^2_{1,2})}\big).
\end{align*}
and
\bbal
||\int^t_0e^{(t-\tau)\De}(\pa_iu\pa_i v)\dd \tau||_{X_T\cap L^\infty_T(BMO)}\leq C||u||_{X_T}||v||_{X_T},
\end{align*}
where
\bbal
||u||_{X_T}=\sup\limits_{0<t\leq T}t^{\frac12}||\na u(t,\cd)||_{L^\infty}+\sup_{x\in\R^d,0<R\leq T^{\fr12}}|B(x,R)|^{-\frac12}||\na u(t,y)||_{L^2_{t,y}([0,R^2]\times B(x,R))}.
\end{align*}
\end{lemm}
\begin{proof}
To prove the first inequality, it is enough to prove the following inequality:
\begin{align}\label{eq2.1}
\|\int_0^t e^{(t-\tau)\Delta}fd\tau\|_{\LL^r_T(\FB^{\frac{2}{r}}_{1,2})}\leq C\|f\|_{\LL^1_T(\FB^0_{1,2})},
\end{align}
for all $r\geq 1$. Indeed, setting $f=\pa_iu\pa_i v$ and applying Lemma \ref{le2.1}, we have
\bbal
& \quad ||\int^t_0e^{(t-\tau)\De}(\pa_iu\pa_i v)\dd \tau||_{\LL^\infty_T(\FB^0_{1,2})}+||\int^t_0e^{(t-\tau)\De}(\pa_iu\pa_i v)\dd \tau||_{\LL^1_T(\FB^2_{1,2})}
\\&\leq
C\|\pa_iu\pa_i v\|_{\LL^1_T(\FB^0_{1,2})}
\\&\leq C\big(||\pa_iu||_{\LL^\infty_T(\FB^{-1}_{1,2})}||\pa_iv||_{\LL^1_T(\FB^1_{1,2})}+||\pa_iv||_{\LL^\infty_T(\FB^{-1}_{1,2})}||\pa_iu||_{\LL^1_T(\FB^1_{1,2})}\big)
\\&\leq C\big(||u||_{\LL^\infty_T(\FB^0_{1,2})}||v||_{\LL^1_T(\FB^2_{1,2})}+||v||_{\LL^\infty_T(\FB^0_{1,2})}||u||_{\LL^1_T(\FB^2_{1,2})}\big).
\end{align*}
Now we prove \eqref{eq2.1}. By Young's inequality, we obtain
\bbal
& \quad \|\int_0^t \widehat{e^{(t-\tau)\Delta}\dot{\Delta}_j f}d\tau\|_{L^r_T L^1}
\\&\leq
C \|\int_0^t e^{2^{2j}(t-\tau)}\|\widehat{\dot{\Delta}_j f}\|_{L^1}d\tau\|_{L^r_T}
\\&\leq C 2^{-\frac{2}{r}j}\|\widehat{\dot{\Delta}_j f}\|_{L^1_T L^1}.
\end{align*}
This implies that inequality \eqref{eq2.1} holds.

The proof of the second inequality is essentially coming from \cite{KT01,L16,W11} and we omit it here.
\end{proof}

\begin{lemm}\label{logs}
For any $p,\rho,\sigma\in[1,\infty]$, $q\in[1,\infty)$ and $s>\frac dq$, there exists a constant $C$ depending only on $d,p$ and q, but not on $\rho,\sigma$ such that for $u\in \dot{B}^{\frac dp}_{p,\rho}\cap B^s_{q,\sigma}$, we have
\bbal
\|u\|_{L^\infty}\leq C (1+\|u\|_{\dot{B}^{\frac dp}_{p,\rho}}(\ln(e+\|u\|_{B^s_{q,\sigma}}))^{1-\frac 1\rho}).
\end{align*}
\end{lemm}
\begin{proof}
See Theorem 2.1 in \cite{KOT02}.	
\end{proof}

\begin{lemm}\label{le-hs}
Assume $u_0\in H^s,s>1+\frac d2$, then there exists a positive time $T_{u_0}\geq \frac{1}{4C_s||u_0||_{H^s}}$ such that  Eq. \eqref{H-J} have a local solution $u\in \mathcal{C}([0,T_{u_0});H^s)$. Moreover, we have the following blow criterion
\bbal
\int^{T_{u_0}}_{0}||\na u||^2_{L^\infty}\dd t=+\infty, \qquad  \mathrm{or} \qquad \int^{T_{u_0}}_0||\na u||^2_{\B^0_{\infty,2}}\dd t=+\infty.
\end{align*}
\end{lemm}
\begin{proof}
We begin by formulating a mollified version of the Cauchy problem \eqref{H-J} sa follows:
\bal\label{ll}
\pa_t\J_nu-\De \J_nu=|\na \J_nu|^2, \quad \J_nu|_{t=0}=\J_n u_0.
\end{align}
Therefore \eqref{ll} defines an ODE on $H^s$ and thus has a unique solution $u_n\in \mathcal{C}([0,T_n);H^s)$, We also have $u_n=\J_nu_n$. Applying the operator $D^s$ to both sides of , multiplying both sides of the resulting equation by $D^su_n$ and then integrating over $\R^d$ yields
the following $H^s$ energy of $u_n$ identity:
\bal\begin{split}\label{lzy-e1}
\frac12\frac{\dd}{\dd t}||u_n||^2_{H^s}+||\na u_n||^2_{H^s}&\leq |||\na u_n|^2||_{H^s}||u_n||_{H^s}
\\&\leq C_s||\na u_n||_{H^s}||\na u_n||_{L^\infty}||u_n||_{H^s}
\\&\leq \frac12||\na u_n||^2_{H^s}+C_s||u_n||^4_{H^s}.
\end{split}\end{align}
Solving differential inequality gives
\bbal
||u_n(t)||_{H^s}\leq \frac{||u_0||_{H^s}}{\sqrt{1-2C_s||u_0||^2_{H^s}t}}.
\end{align*}
Setting $T:=\frac{1}{4C_s||u_0||^2_{H^s}}$, we see from  that the solution $u_n$ exists for $0\leq t\leq T$ and
satisfies a solution size bound
\bbal
||u_n(t)||_{H^s}\leq 2||u_0||_{H^s}, \qquad 0\leq t\leq T.
\end{align*}
By a standard compactness argument, we deduce that \eqref{H-J} has a unique solution $u\in \mathcal{C}([0,T];H^s)$.

Now, we prove the corresponding blow up criterion. For simplicity, we just prove the second one. Let us assume that
\bbal
\int^{T_{u_0}}_0||\na u||^2_{B^0_{\infty,2}}\dd t<+\infty.
\end{align*}
Let $u_0\in H^s$ and $u\in \mathcal{C}([0,T_0];H^s)$ be the corresponding solution of Eq. \eqref{H-J}. Then, applying Lemma \ref{logs} and adopting the similar argument as in \eqref{lzy-e1},  we have for all $t\in[0,T_{u_0})$,
\bbal
\fr12\frac{\dd}{\dd t}||u||^2_{H^s}+||\na u||^2_{H^s}&\leq |||\na u|^2||_{H^s}||u||_{H^s}
\\&\leq C_s||\na u||_{H^s}||\na u||_{L^\infty}||u||_{H^s}
\\&\leq \frac12||\na u||^2_{H^s}+C_s||\na u||^2_{L^\infty}||u||^2_{H^s}
\\&\leq \frac12||\na u||^2_{H^s}+C_s||u||^2_{H^s}(1+||\na u||^2_{B^0_{\infty,2}}\ln(e+||u||_{H^s})),
\end{align*}
which leads to
\bbal
||u(t)||_{H^s}\leq ||u_0||_{H^s}\exp\{C_s\int^t_0(1+||\na u||^2_{B^0_{\infty,2}}\ln(e+||u||_{H^s})) \dd\tau \}.
\end{align*}
Using the Gronwall  inequality again, we get for all $t\in[0,T_{u_0})$,
\bbal
||u(t)||_{H^s}\leq (e+||u_0||_{H^s})^{\exp\{C_s\int^t_0 (1+||\na u||^2_{B^0_{\infty,2}}) \dd\tau \}}.
\end{align*}
Therefore, we can extend the solution past time $T_{u_0}$ which contradicts the assumption. This completes the proof of Lemma \ref{le-hs}.

\end{proof}

\section{Proof of the main theorem}

In this section, we will give the details for the proof of the theorem.

Define a smooth function $\chi$ with values in $[0,1]$ which satisfies
\bbal
\chi(\xi)=
\bca
1, \quad \mathrm{if} \ |\xi|\leq \frac14,\\
0, \quad \mathrm{if} \ |\xi|\geq \frac12.
\eca
\end{align*}

For the cases $2<q\leq \infty$. Letting $\chi^{\pm}_{j}(\xi)=\chi(\xi\mp2^j \e)$ for $j\in \Z$, where $\e=(1,1,\cdots,1)$, then we construct a sequence $\{f^N\}^\infty_{N=1}$ by its Fourier transform
\bbal
\widehat{f^N}(\xi)=\frac{\delta}{N^{\frac12}}\sum^{2N}_{j=N}[\chi^+_j(\xi)+\chi^-_{j}(\xi)]
\end{align*}
It is easy to see that
\bbal
f^N=\frac{2\delta}{N^{\frac12}}\sum^{2N}_{j=N}\cos(2^jx\cd \e)\check{\chi}(x).
\end{align*}
Then, we can directly calculate its Besov norm
\bbal
||f^N||_{B^{0}_{\infty,q}}\leq ||f^N||_{\FB^{0}_{1,q}}\leq C\frac{\de}{N^{\frac12}}\big(\sum^{2N}_{j=N}||\chi(x)||^q_{L^1}\big)^{\frac1q}\leq C\delta\frac{1}{N^{\frac12-\frac1q}},
\end{align*}
which implies $||f^N||_{\FB^{0}_{1,2}}\leq C\de$ and
\bbal
||f^N||_{B^{0}_{\infty,q}}\rightarrow 0, \quad N\rightarrow 0.
\end{align*}
According to Lemma \ref{le2.2}, we have for all $t\in[0,T_{f^N})$,
\bal\begin{split}\label{lzy}
||u[f^N]||_{\LL^\infty_t(\FB^0_{1,2})}+||u[f^N]||_{\LL^1_t(\FB^2_{1,2})}&\leq C||f^N||_{\FB^0_{1,2}}
\\& \quad+C||u[f^N]||_{\LL^\infty_t(\FB^0_{1,2})}||u[f^N]||_{\LL^1_t(\FB^2_{1,2})}.
\end{split}\end{align}
Choosing $\de>0$ small enough, we infer from \eqref{lzy} that
\bbal
||u[f^N]||_{\LL^\infty_t(\FB^0_{1,2})}+||u[f^N]||_{\LL^1_t(\FB^2_{1,2})}&\leq 2C||f^N||_{\FB^0_{1,2}}\leq 2C\de.
\end{align*}
Therefore, by Lemma \ref{le-hs}, we deduce that $u[f^N]\in \mathcal{C}([0,\infty);H^s)$. Now, we need to estimate the second iteration. Let us first to rewrite second iteration as follows:
\bbal
A_2(f^N)=\int^t_0e^{(t-\tau)\De}\sum^d_{i=1}[\pa_ie^{\tau\De}f^N]^2\dd \tau.
\end{align*}
By the definition of Besov spaces, we have
\bal\label{lzy-e3}\begin{split}
||\dot{\De}_{-4}[A_2(f^N)]||_{L^\infty}&\geq c\big|\int_{\R^d}\varphi(16\xi)\mathcal{F}[A_2(f^N)](\xi)\dd \xi\big|
\\&\geq c\Big|\int_{\R^d}\int^t_0e^{-(t-\tau)|\xi|^2}\varphi(16\xi)\sum^d_{i=1}\mathcal{F}\Big([\pa_ie^{\tau\De}f^N]^2\Big)(\xi)\dd \tau\dd \xi\Big|.
\end{split}\end{align}
Since  $\mathrm{Supp }\ \chi(\cdot-a)*\chi(\cdot-b)\subset B(a+b,1)$, we see that
\beq\begin{split}\label{lzy-1}
&\mathrm{Supp} \ \chi^\mu_\ell*\chi^\mu_m\cap B(0,\frac12)=\emptyset, \quad  N\leq \ell,m\leq 2N, \ \mu\in\{+,-\},
\\&\mathrm{Supp} \ \chi^+_\ell*\chi^-_m\cap B(0,\frac12)=\emptyset, \quad N\leq \ell,m\leq 2N, \ \ell\neq m.
\end{split}\eeq
According to \eqref{lzy-1}, we have
\bbal
J(t,\xi)&:=\varphi(16\xi)\sum^d_{i=1}\mathcal{F}\Big([\pa_ie^{\tau\De}f^N]^2\Big)(\xi)
\\&=\frac{-\de^2}{N}\varphi(16\xi)\sum^d_{i=1}\int_{\R^d}\eta_i(\xi_i-\eta_i)e^{-\tau(|\xi-\eta|^2+|\eta|^2)}
\\&\quad \times\sum^{2N}_{j=N}[\chi^+_j(\xi-\eta)+\chi^-_j(\xi-\eta)]\sum^{2N}_{j=N}[\chi^+_j(\eta)+\chi^-_j(\eta)]\dd\eta
\\&=\frac{-2\de^2}{N}\varphi(16\xi)\sum^{2N}_{j=N}\sum^d_{i=1}\int_{\R^d}\eta_i(\xi_i-\eta_i)e^{-\tau(|\xi-\eta|^2+|\eta|^2)}
\chi^+_j(\xi-\eta)\chi^-_j(\eta)\dd\eta.
\end{align*}
Therefore, using the Fubini theorem, we have
\bal\label{lzy-e2}\begin{split}
&\quad I(t,\xi):=\int^t_0e^{-(t-\tau)|\xi|^2}J(\xi)\dd \tau
\\&=\frac{-2\de^2}{N}\varphi(16\xi)\sum^{2N}_{j=N}\sum^d_{i=1}\int_{\R^d}\eta_i(\xi_i-\eta_i)\chi^+_j(\xi-\eta)\chi^-_j(\eta)\Big(\int^t_0
e^{-t|\xi|^2-\tau(|\xi-\eta|^2+|\eta|^2-|\xi|^2)}\dd \tau\Big)\dd \eta
\\&=\frac{-2\de^2}{N}\varphi(16\xi)e^{-t|\xi|^2}\sum^{2N}_{j=N}\sum^d_{i=1}\int_{\R^d}\eta_i(\xi_i-\eta_i)\chi^+_j(\xi-\eta)\chi^-_j(\eta)
\frac{1-e^{-t(|\eta|^2+|\xi-\eta|^2-|\xi|^2)}}{|\eta|^2+|\xi-\eta|^2-|\xi|^2}\dd \eta
\\&=\frac{\de^2}{N}\varphi(16\xi)e^{-t|\xi|^2}\sum^{2N}_{j=N}\int_{\R^d}\chi^+_j(\xi-\eta)\chi^-_j(\eta)(1-e^{2t\eta\cd(\xi-\eta)})\dd \eta
\end{split}\end{align}
Making a change of variable, we can infer form \eqref{lzy-e2} that
\bbal
I(t,\xi)=\frac{\de^2}{N}\varphi(16\xi)e^{-t|\xi|^2}\sum^{2N}_{j=N}\int_{\R^d}\chi(\xi-\eta)\chi(\eta)(1-e^{2t(\eta-2^j\e)\cd(\xi-\eta+2^j\e)})\dd \eta.
\end{align*}
Letting $t=2^{-N}$, then we have $1-e^{2t(\eta-2^j\e)\cd(\xi-\eta+2^j\e)}\geq \frac12$ and
\bal\label{lzy-e4}
|I(t,\xi)|\geq c\de^2\varphi(16\xi)\int_{\R^d}\chi(\xi-\eta)\chi(\eta)\dd \eta\geq c\de^2\varphi(16\xi), \quad \mathrm{for} \  \mathrm{all} \quad |\xi|\leq \frac14.
\end{align}
Combining \eqref{lzy-e3} and \eqref{lzy-e4}, we show that for $t=2^{-N}$,
\bbal
&\qquad\Big|\int_{\R^d}\int^{t}_0e^{-(t-\tau)|\xi|^2}\varphi(16\xi)\sum^d_{i=1}\mathcal{F}\Big([\pa_ie^{\tau\De}f^N]^2\Big)(\xi)\dd \tau\dd \xi\Big|
\\&=\int_{\R^d}\int^{t_N}_0e^{-t_N|\xi|^2}\varphi(16\xi)|I(\xi)|\dd \xi\geq c\de^2,
\end{align*}
which yields to
\bbal
||\dot{\De}_{-4}[A_2(f^N)](t)||_{L^\infty}\geq c\de^2.
\end{align*}
For any $T\geq 1$, we define
\bbal
D=\FB^0_{1,2}, \qquad S=\LL^\infty_T(\FB^0_{1,2})\cap\LL^1_T(\FB^2_{1,2}).
\end{align*}
It is easy to see that
\bbal\bca
\pa_t(u[f^N]-A_1(f^N))-\De(u[f^N]-A_1(f^N))=|\na u[f^N]|^2,
\\ u[f^N]-A_1(f^N)|_{t=0}=0,
\eca\end{align*}
which along with Lemma \ref{le2.2}  yields
\bal\label{lzy-e5}
||u[f^N]-A_1(f^N)||_{S}\leq C||u[f^N]||^2_{S}  \leq C||f^N||^2_D\leq C\de^2.
\end{align}
Using the similar argument as in \eqref{lzy-e5}, it shows that
\bbal
||u[f^N]-A_1(f^N)-A_2(f^N)||_{S}&\leq C(||u[f^N]||_S+||A_1(f^N)||_S)||u[f^N]-A_1(f^N)||_{S}
\\&\leq C\de^3.
\end{align*}
Then, according to the definition of the Besov norm, we have
\bbal
||u[f^N](t_N)||_{B^0_{\infty,q}}&\geq c||\dot{\De}_{-4}\Big(u[f^N](t_N)\Big)||_{L^\infty}
\\&\geq c||\dot{\De}_{-4}\Big(A_2(f^N)(t_N)\Big)||_{L^\infty}
\\& \quad -||\dot{\De}_{-4}\Big(u[f^N]-A_1(f^N)-A_2(f^N)\Big)||_{L^\infty}.
\end{align*}
Therefore, choosing sufficient small $\de>0$ and large enough $N$, it follows that
\bbal
||u[f^N](t_N)||_{B^0_{\infty,q}}&\geq ||\dot{\De}_{-4}[A_2(f^N)](t_N)||_{L^\infty}-||u[f^N]-A_1(f^N)-A_2(f^N)||_{S}
\\&\geq c\de^2-C\de^3\geq c\de^2.
\end{align*}
This completes the proof of the cases $2<q\leq \infty$.

Now, we will prove the theorem for the cases $1\leq q\leq 2$. We set $\chi^{\pm}_{j}(\xi)=\chi(\xi\mp2^j \ee)$ for $j\in \Z$, where $\ee=\frac{17}{24\sqrt{d}}(1,1,\cdots,1)$. Letting $N\in 16\mathbb{N}=\{16,32,48,\cdots\}$ and $\mathbb{N}(N)=\{k\in 8\mathbb{N}: \frac14 N\leq k\leq \frac12 N\}$, we can define a sequence $\{f_N\}^\infty_{N=1}$ by its Fourier transform
\bbal
\widehat{f_N}(\xi)=\frac{\de}{N^{\frac{1}{2q}}}\sum_{\ell\in \mathbb{N}(N)}[\widehat{\Phi^{++}_\ell}+\widehat{\Phi^{+-}_\ell}+\widehat{\Phi^{-+}_\ell}+\widehat{\Phi^{--}_\ell}],
\end{align*}
where
\bbal
&\widehat{\Phi^{++}_\ell}=e^{i2^{\ell+1}\xi\cd \ee}\chi^+_N(\xi-2^\ell\ee), \quad  \widehat{\Phi^{+-}_\ell}=e^{i2^{\ell+1}\xi\cd \ee}\chi^+_N(\xi+2^\ell\ee)
\\&\widehat{\Phi^{-+}_\ell}=e^{i2^{\ell+1}\xi\cd \ee}\chi^-_{N}(\xi-2^\ell\ee), \quad \widehat{\Phi^{--}_\ell}=e^{i2^{\ell+1}\xi\cd \ee}\chi^-_{N}(\xi+2^\ell\ee).
\end{align*}
Here, we have
\bbal
&\Phi^{++}_\ell=e^{i(x+2^{\ell+1}\ee)\cd(2^N\ee+2^{\ell}\ee)}\check{\chi}(x+2^{\ell+1}\ee), \quad \Phi^{+-}_\ell=e^{i(x+2^{\ell+1}\ee)\cd(2^N\ee-2^{\ell}\ee)}\check{\chi}(x+2^{\ell+1}\ee),
\\&\Phi^{-+}_\ell=e^{i(x+2^{\ell+1}\ee)\cd(-2^N\ee+2^{\ell}\ee)}\check{\chi}(x+2^{\ell+1}\ee), \quad \Phi^{--}_\ell=e^{i(x+2^{\ell+1}\ee)\cd(-2^N\ee-2^{\ell}\ee)}\check{\chi}(x+2^{\ell+1}\ee).
\end{align*}
It is easy to see that
\bbal
f_N&=\frac{2\de}{N^{\frac{1}{2q}}}\sum_{\ell\in \mathbb{N}(N)}\big[\cos(x+2^{\ell+1}\ee)\cd(2^N\ee+2^{\ell}\ee)
\\& \quad+\cos(x+2^{\ell+1}\ee)\cd(2^N\ee-2^{\ell}\ee)\big]\check{\chi}(x+2^{\ell+1}\ee).
\end{align*}
Since $\check{\chi}$ is a Schwartz funtion, we have
\bbal
|\check{\chi}(x)|\leq C(1+|x|)^{-M}, \qquad  M\gg 1.
\end{align*}
Then, we have
\bbal
||f_N||_{B^{0}_{\infty,q}}&\leq ||f_N||_{B^{0}_{\infty,1}}\leq C ||f_N||_{L^\infty}\leq C\frac{\de}{N^{\frac{1}{2q}}}\Big|\Big|\sum_{\ell\in \mathbb{N}(N)}\check{\chi}(x+2^{\ell+1}\ee)\Big|\Big|_{L^\infty}
\\&\leq C\frac{\de}{N^{\frac{1}{2q}}}\Big|\Big|\sum_{\ell\in \mathbb{N}(N)}\frac{1}{(1+|x+2^{\ell+1}\ee|)^M}\Big|\Big|_{L^\infty}\leq C\frac{\de}{N^{\frac{1}{2q}}}\rightarrow 0, \quad N\rightarrow \infty.
\end{align*}
According to Lemma \ref{le2.2}, we have for all $t\in[0,T_{f_N})$,
\bal\label{lzy-e6}
||u[f_N]||_{X_t}\leq ||f_N||_{BMO}+||u[f_N]||^2_{X_t}.
\end{align}
Choosing $\de>0$ small enough, we infer from \eqref{lzy-e6} that for all $t\in[0,T_{f_N})$,
\bbal
||u[f_N]||_{X_t}\leq 2C||f_N||_{BMO}\leq 2C\frac{\de}{N^{\frac{1}{2q}}}.
\end{align*}
Therefore, by Lemma \ref{le-hs}, we deduce that $u[f^N]\in \mathcal{C}([0,\infty);H^s)$.

Next, we need to estimate the second iteration in the Besov norm. To complete our goal, we will use the following lemmas to hackle with.
\begin{lemm}\label{le-e1}
Let $1\leq q\leq 2$. Then, there exists a positive constant $c$ independent of $N$ and $\de$ such that
\bbal
||(\pa_if_N)^2||_{B^0_{\infty,q}(\mathbb{N}(N))}\geq c\de^2 2^{2N}, \qquad N\gg1.
\end{align*}
\end{lemm}
\begin{proof}
Since $\mathrm{Supp }\ \chi(\cdot-a)*\chi(\cdot-b)\subset B(a+b,1)$, we see that
\bbal
\mathrm{Supp} \ \widehat{\Phi^{\nu\lambda}_\ell}*\widehat{\Phi^{\nu\mu}_m}\subset B(\nu2^{N+1}\ee,2^{2+\frac N2}), \quad \lambda,\mu,\nu\in\{+,-\}.
\end{align*}
Moreover, noticing that Supp $\widehat{\Phi^{++}_\ell}*\widehat{\Phi^{--}_m}\subset B((2^\ell-2^m)\e,1)$, we see that Supp $\widehat{\Phi^{++}_\ell}*\widehat{\Phi^{--}_\ell}\subset B(0,1)$. It follows that for any $j\in \mathbb{N}(N)$,
\bal\label{l}\begin{split}
&\quad \De_j[(\pa_i f_N)^2]\\&=\frac{8\de^2}{N^{\frac1q}}\De_j\sum_{\ell,m\in \mathbb{N}(N)}(\pa_i\Phi^{++}_\ell\pa_i\Phi^{-+}_m+\pa_i\Phi^{++}_\ell\pa_i\Phi^{--}_m
+\pa_i\Phi^{+-}_\ell\pa_i\Phi^{-+}_m
+\pa_i\Phi^{+-}_\ell\pa_i\Phi^{--}_m)
\end{split}\end{align}
Note that if $\max\{\ell,m\}>j,\ \ell,m \in  \mathbb{N}(N), \ \ell\neq m$, we have
\bal\label{l1}\begin{split}
&\mathrm{Supp} \ \widehat{\Phi^{++}_\ell}*\widehat{\Phi^{-+}_m}\subset B((2^\ell+2^m)\ee,1)\subset 2^{j}\mathcal{C}(0,3,2^N),
\\&\mathrm{Supp} \ \widehat{\Phi^{++}_\ell}*\widehat{\Phi^{--}_m}\subset B((2^\ell-2^m)\ee,1)\subset 2^{j}\mathcal{C}(0,3,2^N),
\\&\mathrm{Supp} \ \widehat{\Phi^{+-}_\ell}*\widehat{\Phi^{-+}_m}\subset B((-2^\ell+2^m)\ee,1)\subset 2^{j}\mathcal{C}(0,3,2^N),
\\&\mathrm{Supp} \ \widehat{\Phi^{+-}_\ell}*\widehat{\Phi^{--}_m}\subset B((-2^\ell-2^m)\ee,1)\subset 2^{j}\mathcal{C}(0,3,2^N).
\end{split}\end{align}
Else if $\max\{\ell,m\}<j,\ \ell,m\in \mathbb{N}(N),\ \ell\neq m$, we have
\bal\label{l2}\begin{split}
&\mathrm{Supp} \ \widehat{\Phi^{++}_\ell}*\widehat{\Phi^{-+}_m}\subset B((2^\ell+2^m)\ee,1)\subset 2^{j}B(0,\frac12),
\\&\mathrm{Supp} \ \widehat{\Phi^{++}_\ell}*\widehat{\Phi^{--}_m}\subset B((2^\ell-2^m)\ee,1)\subset 2^{j}B(0,\frac12),
\\&\mathrm{Supp} \ \widehat{\Phi^{+-}_\ell}*\widehat{\Phi^{-+}_m}\subset B((-2^\ell+2^m)\ee,1)\subset 2^{j}B(0,\frac12),
\\&\mathrm{Supp} \ \widehat{\Phi^{+-}_\ell}*\widehat{\Phi^{--}_m}\subset B((-2^\ell-2^m)\ee,1)\subset 2^{j}B(0,\frac12).
\end{split}\end{align}
Moreover, using the facts that
\bbal
&\mathrm{Supp} \ \widehat{\Phi^{++}_j}*\widehat{\Phi^{-+}_j}\subset B(2^{j+1}\ee,1)\subset 2^j\mathcal{C}(0,\frac43,\frac32),
\\&\mathrm{Supp} \ \widehat{\Phi^{++}_j}*\widehat{\Phi^{--}_j}\subset B(0,1),
\\&\mathrm{Supp} \ \widehat{\Phi^{+-}_j}*\widehat{\Phi^{-+}_j}\subset B(0,1),
\\&\mathrm{Supp} \ \widehat{\Phi^{+-}_j}*\widehat{\Phi^{--}_j}\subset B(2^{j+1}\ee,1)\subset 2^j\mathcal{C}(0,\frac43,\frac32),
\end{align*}
we have
\bal\label{l3}\begin{split}
&\quad \De_j(\pa_i\Phi^{++}_j\pa_i\Phi^{-+}_j+\pa_i\Phi^{++}_j\pa_i\Phi^{--}_j+\pa_i\Phi^{+-}_j\pa_i\Phi^{-+}_j+\pa_i\Phi^{+-}_j\pa_i\Phi^{--}_j)
\\&=\De_j(\pa_i\Phi^{++}_j\pa_i\Phi^{-+}_j+\pa_i\Phi^{+-}_j\pa_i\Phi^{--}_j)
\\&=\pa_i\Phi^{++}_j\pa_i\Phi^{-+}_j+\pa_i\Phi^{+-}_j\pa_i\Phi^{--}_j.
\end{split}\end{align}
Therefore, plugging \eqref{l1}-\eqref{l3} into \eqref{l}, we have
\bbal
&\quad \De_j[(\pa_i f_N)^2]
\\&=\frac{8\de^2}{N^{\frac1q}}(\pa_i\Phi^{++}_j\pa_i\Phi^{-+}_j+\pa_i\Phi^{+-}_j\pa_i\Phi^{--}_j)
\\& \quad+\frac{8\de^2}{N^{\frac1q}}\De_j\sum_{k\leq j-1,k\in \mathbb{N}(N)}(\pa_i\Phi^{++}_j\pa_i\Phi^{-+}_k+\pa_i\Phi^{++}_k\pa_i\Phi^{-+}_j)
\\& \quad+\frac{8\de^2}{N^{\frac1q}}\De_j\sum_{k\leq j-1,k\in \mathbb{N}(N)}(\pa_i\Phi^{++}_j\pa_i\Phi^{--}_k+\pa_i\Phi^{++}_k\pa_i\Phi^{--}_j)
\\& \quad+\frac{8\de^2}{N^{\frac1q}}\De_j\sum_{k\leq j-1,k\in \mathbb{N}(N)}(\pa_i\Phi^{+-}_j\pa_i\Phi^{-+}_k+\pa_i\Phi^{+-}_k\pa_i\Phi^{-+}_j)
\\& \quad+\frac{8\de^2}{N^{\frac1q}}\De_j\sum_{k\leq j-1,k\in \mathbb{N}(N)}(\pa_i\Phi^{+-}_j\pa_i\Phi^{--}_k+\pa_i\Phi^{+-}_k\pa_i\Phi^{--}_j)
\\&=I_1+I_2+I_3+I_4+I_5.
\end{align*}
Note that
\bbal
\frac{N^{\frac1q}}{8\de^2}I_1&=\pa_i\Phi^{++}_j\pa_i\Phi^{-+}_j+\pa_i\Phi^{+-}_j\pa_i\Phi^{--}_j
\\&=(2^{2N}-2^{2j})[\Phi^{++}_j\Phi^{-+}_j+\Phi^{+-}_j\Phi^{--}_j]+2\cos[(x+2^{j+1}\ee)\cd 2^{j+1}\e][\pa_i\check{\chi}(x+2^{j+1}\ee)]^2
\\&=\frac{289}{288d}(2^{2N}-2^{2j})\cos[(x+2^{j+1}\ee)\cd 2^{j+1}\ee]\check{\chi}^2(x+2^{j+1}\ee)
\\&\quad +2\cos[(x+2^{j+1}\ee)\cd 2^{j+1}\ee][\pa_i\check{\chi}(x+2^{j+1}\ee)]^2.
\end{align*}
By choosing $x=-2^{j+1}\ee$ and using $\mathcal{F}^{-1}\big(\xi_i\chi\big)(0)=0$, we have
\bal\label{y1}
||I_1||_{L^\infty}\geq c\frac{\delta^2}{N^{\frac1q}}2^{2N}\check{\chi}^2(0).
\end{align}
Since Supp $\Phi^{\mu,\nu}_k\subset 2^N\mathcal{C}, k\in \mathbb{N}(N), \mu,\nu\in\{+,-\}$, we have
\bal\label{y2}\begin{split}
&\quad \ ||I_2||_{L^\infty}\\&\leq C\frac{\delta^2}{N^{\frac1q}}\Big|\Big|\sum_{k\leq j-1,k\in \mathbb{N}(N)}(\pa_i\Phi^{++}_j\pa_i\Phi^{-+}_k+\pa_i\Phi^{++}_k\pa_i\Phi^{-+}_j)\Big|\Big|_{L^\infty}
\\&\leq C\frac{\delta^2}{N^{\frac1q}}\Big|\Big|\sum_{k\leq j-1,k\in \mathbb{N}(N)}2^{2N}(1+|x+2^N\ee+2^{j}\ee|^2)^{-M}(1+|x-2^N\ee+2^{k}\ee|^2)^{-M}\Big|\Big|_{L^\infty}
\\&\leq C\frac{\delta^2}{N^{\frac1q}}2^{2N}j2^{-MN}\leq C\frac{\delta^2}{N^{\frac1q}}2^{2N}N2^{-NM}.
\end{split}\end{align}
Applying similar estimate as in \eqref{y2}, we have
\bal\label{y3}
||I_i||_{L^\infty}\leq C\frac{\delta^2}{N^{\frac1q}}2^{2N}N2^{-NM}, \quad \mathrm{for} \quad i=3,4,5.
\end{align}
Combining \eqref{y1}-\eqref{y3}, we deduce that for $N\gg1$,
\bbal
||\De_j[(\pa_i f_N)^2]||_{L^\infty}\geq c\frac{\de^2}{N^{\frac1q}}2^{2N}.
\end{align*}
Therefore, by the definition of the Besov norm, we have
\bbal
||(\pa_if_N)^2||_{B^0_{\infty,q}(\mathbb{N}(N))}\geq \Big(\sum_{j\in \mathrm{N}(N)}||\De_j[(\pa_if_N)^2]||^q_{L^\infty}\Big)^{\frac1q}\geq c\de^22^{2N}.
\end{align*}
This completes the proof of this lemma.
\end{proof}

Now, we can rewrite
\bbal
A_2(f_N)&=\int^t_0e^{(t-\tau)\De}\sum^d_{i=1}[\pa_ie^{\tau\De}f_N]^2\dd \tau
\\&=t\sum^d_{i=1}(\pa_if_N)^2+t\sum_{r\geq 2}\frac{1}{r!}(t\De)^{r-1}\big[\sum^d_{i=1}(\pa_if_N)^2\big]
\\&\quad +\int^t_0e^{(t-\tau)\De}\sum^d_{i=1}(e^{\tau\De}+1)\pa_if_N(e^{\tau\De}-1)\pa_if_N\dd \tau.
\end{align*}

\begin{lemm}\label{le-e2}
Let $1\leq q\leq 2$ and $t=\de2^{-2N}$. Then there exists a constant $C$ independent of $N$ and $\de$ such that
\bbal
t\sum_{r\geq 2}\frac{1}{r!}||(t\De)^{r-1}\big[(\pa_if_N)^2\big]||_{B^0_{\infty,q}(\mathbb{N}(N))}\leq C\de^4.
\end{align*}
\end{lemm}
\begin{proof}
Using the Bernstein's multiplier estimates, we have
\bbal
||\mathcal{F}^{-1}\big((t|\xi|^2)^{r-1}\varphi_j\mathcal{F}f)\big)||_{L^\infty}&\lesssim (t2^{2j})^{r-1}||\mathcal{F}^{-1}(|\xi|^{2(r-1)}\varphi)||_{L^1}||f||_{L^\infty}
\\&\lesssim (t2^{2j})^{r-1}9^rr^d||f||_{L^\infty}.
\end{align*}
Noticing that $j\in \mathbb{N}(N)$ implies that $j\leq \frac N2$, one has that
\bbal
t\sum_{r\geq 2}\frac{1}{r!}||(t\De)^{r-1}\big[(\pa_if_N)^2\big]||_{B^0_{\infty,q}(\mathbb{N}(N))}&\lesssim t\sum_{r\geq 2}\frac{1}{r!}\Big(\sum_{j\in \mathbb{N}(N)}(t2^{2j})^{q(r-1)}||\pa_if_N||^{2q}_{L^\infty}\Big)^{\frac1q}
\\&\lesssim \de^3\sum_{r\geq 2}\frac{9^rr^d}{r!}N^{\frac1q}N^{-\frac1q}(t2^N)^{r-1}
\\&\lesssim \de^4,
\end{align*}
which completes the proof of this lemma.
\end{proof}

\begin{lemm}\label{le-e3}
Let $1\leq q\leq 2$, $\de\ll1$, and $t=\de2^{-2N}$. Then there exists a constant $C$ independent of $N$ and $\de$ such that
\bbal
||\int^t_0e^{(t-\tau)\De}(e^{\tau\De}+1)\pa_if_N(e^{\tau\De}-1)\pa_if_N\dd \tau||_{B^{0}_{\infty,q}(\mathbb{N}(N))}\leq C\de^4.
\end{align*}
\end{lemm}
\begin{proof}
Using the fact that $||e^{t\De}u||_{L^\infty}\leq C||u||_{L^\infty}$, we have
\bal\label{z1}\begin{split}
&\qquad ||\int^t_0e^{(t-\tau)\De}(e^{\tau\De}+1)\pa_if_N(e^{\tau\De}-1)\pa_if_N\dd \tau||_{B^{0}_{\infty,q}(\mathbb{N}(N))}\\&\lesssim t2^{2N}N^\frac1q||f_N||_{L^\infty} \max_{\tau\in(0,t]}||(e^{\tau\De}-1)f_N||_{L^\infty}\lesssim \de^2N^{\frac{1}{2q}}\max_{\tau\in(0,t]}||(e^{\tau\De}-1)f_N||_{L^\infty}.
\end{split}\end{align}
Using Taylor's expansion, one has that
\bbal
||(e^{\tau\De}-1)f_N||_{L^\infty}\leq \sum_{r\geq1}\frac{\tau^r}{r!}||\mathcal{F}^{-1}\big(|\xi|^{2r}\mathcal{F}(f_N)\big)||_{L^\infty}.
\end{align*}
Since Supp $\widehat{f_N}\subset 2^N\mathcal{C}$, we see that
\bal\label{z2}\begin{split}
||\mathcal{F}^{-1}\big(|\xi|^{2r}\mathcal{F}(f_N)\big)||_{L^\infty}&\lesssim ||\mathcal{F}^{-1}(|\xi|^{2r}\varphi_N)||_{L^1}||f_N||_{L^\infty}
\\&\lesssim \de9^rr^d2^{2Nr}\lesssim \de (18)^r2^{2kr}.
\end{split}\end{align}
It follows from \eqref{z2} that
\bal\label{z3}\begin{split}
||(e^{\tau\De}-1)f_N||_{L^\infty}&\lesssim \de N^{-\frac{1}{2q}}[\sum_{r\geq 1} \frac{(18\tau2^{2N})^r}{r!}-1]\lesssim \de N^{-\frac{1}{2q}}(e^{18\tau2^{2N}}-1)
\\&\lesssim \de N^{-\frac{1}{2q}}(e^{18\de}-1)\lesssim \de^2N^{-\frac{1}{2q}}.
\end{split}\end{align}
Plugging \eqref{z3} into \eqref{z1}, we have the result.

\end{proof}

Collecting Lemmas \ref{le-e1}-\ref{le-e3}, we immediately have
\begin{lemm}
Let $1\leq q\leq 2$ and $t=\de2^{-2N}$. Then, if we choose $\de$ small enough and $N$ large enough, there exists a positive constant $c$ independent of $N$ and $\de$ such that
\bbal
||A_2(f_N)(t)||_{B^0_{\infty,q}(\mathbb{N}(N))}\geq c\de^3.
\end{align*}
\end{lemm}

Now, we will accomplish our proof of the theorem. For any $T\geq 1$, we have
\bbal
||u[f_N]-A_1(f_N)||_{Y_T}&\leq \Big|\Big|\int^t_0e^{(t-\tau)\De}(|\na u[f_N]|^2)\Big|\Big|_{Y_T}
\\&\leq C||u[f_N]||^2_{X_T}\leq C||f_N||^2_{BMO}\leq C\frac{\delta^2}{N^{\frac1q}},
\end{align*}
where
\bbal
||u||_{Y_T}=||u||_{X_T}+||u||_{L^\infty_T(BMO)}.
\end{align*}
On the other hand, we also have
\bbal
||u[f_N]-A_1(f_N)-A_2(f_N)||_{Y_T}&\leq (||u[f_N]||_{X_T}+||A_1(f_N)||_{X_T})||u[f_N]-A_1(f_N)||_{X_T}
\\&\leq ||f_N||_{BMO}||u[f_N]-A_1(f_N)||_{X_T}\leq C\frac{\de^3}{N^{\frac{3}{2q}}},
\end{align*}
which implies
\bbal
||u[f_N]-A_1(f_N)-A_2(f_N)||_{B^0_{\infty,\infty}(\mathbb{N}(N))}\leq C||u[f_N]-A_1(f_N)-A_2(f_N)||_{Y_T}\leq C\frac{\de^3}{N^{\frac{3}{2q}}}.
\end{align*}
Then, by the definition of Besov norm, we show that
\bbal
&\quad ||u[f_N]-A_1(f_N)-A_2(f_N)||_{B^0_{\infty,q}(\mathbb{N}(N))}
\\&\leq CN^{\frac1q}||u[f_N]-A_1(f_N)-A_2(f_N)||_{B^0_{\infty,\infty}(\mathbb{N}(N))}\leq C\frac{\de^3}{N^{\frac{1}{2q}}}.
\end{align*}
Choosing small enough $\de>0$, we have
\bbal
||u[f_N](t)||_{B^0_{\infty,q}}&\geq ||u[f_N]||_{B^0_{\infty,q}(\mathbb{N}(N))}
\\&\geq ||A_2(f_N)(t)||_{B^0_{\infty,q}(\mathbb{N}(N))}-||A_1(f_N)||_{B^0_{\infty,q}(\mathbb{N}(N))}
\\& \qquad -||u[f_N]-A_1(f_N)-A_2(f_N)||_{B^0_{\infty,q}(\mathbb{N}(N))}
\\&\geq c\de^3-C\frac{\de^3}{N^{\frac{1}{2q}}}\geq c\de^3,\qquad N\rightarrow \infty.
\end{align*}

\vspace*{1em}
\noindent\textbf{Acknowledgements.} This work was
partially supported by NNSFC (No. 11271382), RFDP (No. 20120171110014), MSTDF (No. 098/2013/A3), Guangdong Special Support Program (No. 8-2015) and the key project of NSF of Guangdong Province (No. 1614050000014).

\end{document}